\documentclass{article}

\usepackage{amsmath,amssymb,amsthm}
\usepackage[initials]{amsrefs}
\newcommand{\R}{{\mathbb{R}}}

\newtheorem{theorem}{Theorem}
\newtheorem{proposition}{Proposition}

\newcommand{\graph}{{\mathcal{G}}}
\newcommand{\verts}{{\mathcal{V}}}
\newcommand{\edges}[1][]{{\mathcal{E}}}
\newcommand{\edgein}{{\mathcal E}_+}
\newcommand{\edgeout}{{\mathcal E}_{-}}
\newcommand{\spn}{\operatorname{span}}
\usepackage{tikz} \usetikzlibrary{arrows}

\begin{document}
\title{Finiteness of the H\"{o}lder-Brascamp-Lieb Constant Revisited}
\author{Philip T. Gressman\footnote{Partially supported by NSF grant DMS-2348384.}}
\maketitle
\abstract{Abstract H\"{o}lder-Brascamp-Lieb inequalities have become a ubiquitous tool in Fourier analysis in recent years, due in large part to a theorem of Bennett, Carbery, Christ, and Tao \cites{bcct2008,bcct2010} characterizing finiteness of the H\"{o}lder-Brascamp-Lieb constant. Here we provide a new characterization of a substantially different nature involving directed graphs of subspaces. Its practical value derives from its complementary nature to the Bennett, Carbery, Christ, and Tao conditions: it creates a means by which one can establish finiteness of the H\"{o}lder-Brascamp-Lieb constant by analysis of a well-chosen, finite list of subspaces rather than by checking conditions on all subspaces of the underlying vector space. The proof is elementary and is essentially an ``explicitization'' of the semi-explicit factorization algorithm of Carbery, H\"{a}nninen, and Valdimarsson \cite{chv2023}.}

\section{Introduction}
Suppose that $H$ is some finite-dimensional real Hilbert space and that $\{\pi_i\}_{i=1}^N$ is a collection of real linear maps with each $\pi_i$ taking values in some real Hilbert space $H_i$ for $i=1,\ldots,N$. Suppose also that $\{\tau_i\}_{i=1}^N$ is a collection of real numbers in the interval $[0,1]$. The full collection of objects $(H,\{\pi_i\}_{i=1}^N,\{\tau_i\}_{i=1}^N)$ is called H\"{o}lder-Brascamp-Lieb data. Given such data, one commonly seeks to determine whether there exists a finite constant $C$  such that the inequality
\begin{equation} \int_{H} \prod_{i=1}^N [f_i(\pi_i(x))]^{\tau_i} dx \leq C \prod_{i=1}^N \left[ \int_{\pi_i(H)} f_i(x_i) dx_i \right]^{\tau_i} \label{brascamplieb} \end{equation}
holds for all nonnegative Lebesgue-measurable functions $f_i$ on $\pi_i(H)$. This general family of inequalities is immensely important in analysis, containing within it, for example, H\"{o}lder's inequality, Young's inequality for convolutions, and the Loomis-Whitney inequality. 

Lieb \cite{lieb1990} established that the optimal value of $C$ can be determined by testing the inequality on the family of centered Gaussian functions. Because the underlying family is noncompact, however, it can be challenging to carry out Lieb's computation explicitly, especially in cases of low symmetry. Thus, it is useful to have simpler criteria by which one may answer the more basic question of whether \textit{any} finite $C$ exists satisfying \eqref{brascamplieb}. The finiteness question was first answered in influential work of Bennett, Carbery, Christ, and Tao \cites{bcct2008,bcct2010} which generalized special cases studied by Carlen, Lieb, and Loss \cite{cll2004} and others. In particular, they established that $C$ is finite if and only if two conditions hold. First, every subspace $V$ of $H$ must satisfy the inequality
\begin{equation} \dim V \leq \sum_{i=1}^N \tau_i \dim \pi_i(V). \label{bcct1} \end{equation}
Second, when $V = H$,
\begin{equation} \dim H =  \sum_{i=1}^N \tau_i \dim \pi_i(H). \label{bcct2} \end{equation}
This characterization of finiteness, along with various special cases and closely-related results, has played a role in an array of results from the past fifteen years \cites{bct2006,bdg2016,gozz2023,bt2024,mo2024,gz2020}.

By connecting Lieb's formula to Geometric Invariant Theory, it is possible to deduce that there must be alternate characterizations of the finiteness of the H\"{o}lder-Brascamp-Lieb constant which are of a very different nature than the Bennett-Carbery-Christ-Tao conditions; see \cite{gressman2021}*{Lemma 2}. Whereas \eqref{bcct1} must be systematically verified for all subspaces $V$ to deduce finiteness of the H\"{o}lder-Brascamp-Lieb constant, algebraic approaches can establish finiteness of $C$ when one is able to cleverly construct just one invariant polynomial which does not annihilate the H\"{o}lder-Brascamp-Lieb data. (And conversely: a single, well-chosen subspace $V$ violating \eqref{bcct1} is enough to prove that the H\"{o}lder-Brascamp-Lieb constant is infinite, while algebraic methods require verifying that all invariant polynomials annihilate the H\"{o}lder-Brascamp-Lieb maps.) Such an alternate characterization is valuable from the pragmatic standpoint of reducing the labor involved in validating \eqref{brascamplieb}. The chief drawback of the approach \cite{gressman2021} is simply that it gives no genuine understanding of how to find invariants. The present paper gives one way to address this deficiency by building a descriptive framework for a family of explicit computations. The result is a characterization of the finiteness of the H\"{o}lder-Brascamp-Lieb constant which is complementary to the existing one.

The basic computational unit of this new characterization is given in terms of certain highly-structured directed graphs.
Supposing, as above, that $H$ is some finite-dimensional real Hilbert space, a \textit{graph decomposition} of $H$ is any directed graph $\graph$ whose vertex set $\verts(\graph)$ consists of subspaces of $H$, where at minimum both $\{0\}$ and $H$ belong to $\verts(\graph)$, and whose edge set $\edges(\graph)$ satisfies the the following constraints:
\begin{enumerate}
\item For every edge $e \in \edges$ which is outgoing from $V_1$ and incoming to $V_2$, $V_1 \subset V_2$ and $\dim V_2 = \dim V_1 + 1$.
\item Every vertex other than $\{0\}$ has at least one incoming edge.
\item Every vertex other than $H$ has at least one outgoing edge.
\end{enumerate}
A scalar- or vector-valued function $\varphi$ on the edge set $\edges(\graph)$ is called \textit{balanced} when for all vertices $V \in \verts$ which are not equal to $H$ or $\{0\}$, the sum of $\varphi(e)$ over all edges $e$ incoming to $V$ equals the sum of $\varphi(e')$ over all edges $e'$ outgoing from $V$.
The \textit{total mass} of $\varphi$ is defined to be sum of $\varphi(e)$ over all edges $e$ outgoing from $\{0\}$.
Given a collection of maps $\{\pi_i\}_{i=1}^N$, a graph decomposition $\graph$ of $H$, and nonnegative balanced weights $\theta_1,\ldots,\theta_N$ on $\graph$ such that the total mass of $\theta_i$ is $\tau_i$ for each $i$, we define the \textit{summary weight} $\sigma$  on each edge $e$ as the sum of $\theta_i(e)$ over all $i$ such that the endpoints of $e$ project via $\pi_i$ to unequal subspaces. When the summary weights $\sigma$ are balanced with total mass $1$, we say that the collection $(\graph,\{\theta_i\}_{i=1}^N)$ is a \textit{valid presentation} of the H\"{o}lder-Brascamp-Lieb data $(H,\{\pi_i\}_{i=1}^N, \{\tau_i\}_{i=1}^N)$.
The main theorem is stated as follows.
\begin{theorem}
For given data $(H,\{\pi_i\}_{i=1}^N, \{\tau_i\}_{i=1}^N)$, the H\"{o}lder-Brascamp-Lieb inequality \eqref{brascamplieb} holds for some finite constant $C$ and all nonnegative measurable functions $\{f_i\}_{i=1}^N$ if and only if there exists $(\graph,\{\theta_i\}_{i=1}^N)$ which is a valid presentation of the data. \label{mainthm}
\end{theorem}

The proof of Theorem \ref{mainthm} will be accomplished via an elementary factorization argument. As part of their remarkable theory of multilinear duality, Carbery, H\"{a}nninen, and Valdimarsson \cite{chv2023} observed that factorization methods can be used to prove H\"{o}lder-Brascamp-Lieb inequalities when the conditions \eqref{bcct1} and \eqref{bcct2} are known to hold. The approach they use is to inductively develop ``semi-explicit'' factorizations by reducing matters to critical subspaces and interpolation. The factorization produced by their algorithm is (essentially by necessity) essentially the same as the factorization used here. The main distinction between the algorithm of Carbery, H\"{a}nninen, and Valdimarsson and Theorem \ref{mainthm} is that here the construction of the factorization is justified by quite different reasoning which avoids the need for \textit{a priori} validation of \eqref{bcct1} and \eqref{bcct2}. The organization of the proof of Theorem \ref{mainthm} is also (again, by necessity) somewhat different than their approach as well, since in the main direction (proving finiteness of the H\"{o}lder-Brascamp-Lieb constant), critical subspaces play no distinguished role in the new argument. Critical subspaces do, however, play an important role in the direction of showing that the Bennett-Carbery-Christ-Tao conditions imply the existence of a valid presentation $(\graph,\{\theta_i\}_{i=1}^N)$, and it is in this  direction that one sees most clearly the familiar structural properties of H\"{o}lder-Brascamp-Lieb inequalities which are exploited by both Bennet, Carbery, Christ, and Tao \cite{bcct2010} and Carbery, H\"{a}nninen, and Valdimarsson.

Stated differently, the novel content of Theorem \ref{mainthm} is that the Bennett-Carbery-Christ-Tao conditions \eqref{bcct1}, \eqref{bcct2} are equivalent to the existence of a valid presentation $(\graph,\{\theta_i\}_{i=1}^N)$. It seems likely possible to establish this equivalence directly by methods which only make reference to elementary linear algebra, but powerful factorization ideas make it easier to establish the equivalence by taking a ``analytic detour'' reproving the H\"{o}lder-Brascamp-Lieb inequalities.

As a simple example, consider the application of Theorem \ref{mainthm} to the Loomis-Whitney inequality. Here we have an underlying Hilbert space $H := \R^{d+1}$ and maps $\pi_1,\ldots,\pi_{d+1} : \R^{d+1} \rightarrow \R^{d}$ given by
$\pi_i(x) := (x_1,\ldots,\widehat{x_i},\ldots,x_{d+1})$,
where $x := (x_1,\ldots,x_{d+1})$ and $\widehat{\cdot}$ denotes omission. Let $e_1,\ldots,e_{d+1}$ denote the standard basis vectors in $\R^d$. Fixing $V_0 := \{0\}$ and $V_i := \spn \{e_1,\ldots,e_i\}$ for $i=1,\ldots,d+1$, we may build a graph decomposition $\graph$ of $\R^{d+1}$ by joining $V_{i-1}$ to $V_{i}$ by a directed edge (outgoing from $V_{i-1}$ and incoming to $V_i$) for each $i=1,\ldots,d+1$. The balanced weights $\theta_1,\ldots,\theta_{d+1}$ are simply equal to $1/d$ on each edge. The summary weight $\sigma$ equals $1$ on each edge because $V_{i-1}$ and $V_i$ project to the same space via $\pi_{i'}$ if and only if $i = i'$. This information can be conveniently summarized via the following diagram.
\begin{center}
\begin{tikzpicture}[scale=0.865]
\tikzset{vertex/.style={}}
\tikzset{edge/.style={->,>=latex'}}
\node[vertex] (a) at (0,0) {$\{0\}$};
\node[vertex] (b) at (2.2,0) {$V_1$};
\node[vertex] (c) at (4.4,0) {$V_2$};
\node[vertex] (d) at (6.6,0) {$\cdots$};
\node[vertex] (e) at (8.8,0) {$V_{d}$};
\node[vertex] (f) at (11,0) {$\R^{d+1}$};
\draw[edge] (a)--(b) node [midway,above] {$(\frac{1}{d},\underline{\frac{1}{d}},\ldots,\underline{\frac{1}{d}})$};
\draw[edge] (b)--(c) node [midway,above] {$(\underline{\frac{1}{d}},\frac{1}{d},\ldots,\underline{\frac{1}{d}})$};
\draw[edge] (c)--(d) node [midway,above] {$\cdots$};
\draw[edge] (d)--(e) node [midway,above] {$\cdots$};
\draw[edge] (e)--(f) node [midway,above] {$(\underline{\frac{1}{d}},\underline{\frac{1}{d}},\ldots,{\frac{1}{d}})$};
\end{tikzpicture}
\end{center}
Here the labels on each edge are the weights $(\theta_1,\ldots,\theta_{d+1})$. Individual entries are underlined when for the corresponding value of $i'$, $\pi_{i'}$ projects the endpoints of that edge to unequal subspaces. Given such a diagram, verifying that it describes a valid presentation involves only simple and explicit computations.

For a more elaborate example, consider the following collection of maps on $\R^6$, where $x := (x_1,\ldots,x_6)$:
\begin{align*}
\pi_1(x) & := (x_1,x_2,x_5), \\
\pi_2(x) & := (x_2,x_3,x_5+x_6), \\
\pi_3(x) & := (x_4,x_6), \\
\pi_4(x) & := (x_1,x_3,x_4,x_5-x_6).
\end{align*}
Let $\{e_i\}_{i=1}^6$ be the standard basis vectors of $\R^6$ and consider the subspaces
\begin{align*}
V_i & := \spn \{e_1,\ldots,e_i\}, \ i = 1,\ldots,5,\\
V_6 & := \spn \{e_1,e_2,e_3,e_4,e_5 + e_6\}.
\end{align*}
Defining $\tau_1 := \tau_2 := \tau_3 := \tau_4 := 1/2$, the data $(\R^6,\{\pi_i\}_{i=1}^4,\{\tau_i\}_{i=1}^4)$ has a valid presentation which is summarized in the following diagram.
\begin{center}
\begin{tikzpicture}[scale=0.885]
\tikzset{vertex/.style={}}
\tikzset{edge/.style={->,>=latex'}}
\node[vertex] (a) at (0,0) {$\{0\}$};
\node[vertex] (b) at (2,0) {$V_1$};
\node[vertex] (c) at (4,0) {$V_2$};
\node[vertex] (d) at (6,0) {$V_3$};
\node[vertex] (e) at (8,0) {$V_4$};
\node[vertex] (f) at (9.414,1.414) {$V_5$};
\node[vertex] (g) at (9.414,-1.414) {$V_6$};
\node[vertex] (h) at (10.828,0) {$\R^6$};
\draw[edge] (a)--(b) node [midway, above] () {$(\underline{\frac{1}{2}},\frac{1}{2},\frac{1}{2},\underline{\frac{1}{2}})$};
\draw[edge] (b)--(c) node [midway, above] () {$(\underline{\frac{1}{2}},\underline{\frac{1}{2}},\frac{1}{2},\frac{1}{2})$};
\draw[edge] (c)--(d) node [midway, above] () {$(\frac{1}{2},\underline{\frac{1}{2}},\frac{1}{2},\underline{\frac{1}{2}})$};
\draw[edge] (d)--(e) node [midway, above] () {$(\frac{1}{2},\frac{1}{2},\underline{\frac{1}{2}},\underline{\frac{1}{2}})$};
\draw[edge] (e)--(f) node [midway, above,sloped] () {${ }\, (\underline{\frac{1}{2}},\underline{0},\frac{1}{2},\underline{0})$};
\draw[edge] (e)--(g) node [midway,below,sloped] () {$(\underline{0},\underline{\frac{1}{2}},\underline{0},\frac{1}{2})$};
\draw[edge] (f)--(h) node [midway,above,sloped] () {$(\frac{1}{2},0,\underline{\frac{1}{2}},0)$};
\draw[edge] (g)--(h)  node [midway,below,sloped] () {$(0,\frac{1}{2},0,\underline{\frac{1}{2}})$};
\end{tikzpicture}
\end{center}
The balanced weights $\theta_1,\theta_2,\theta_3,\theta_4$ are listed (in order) along every edge, and the summary weight $\sigma$ on each edge is simply the sum of the underlined values for each edge.

In general, one should not expect that valid presentations for a particular collection of H\"{o}lder-Brascamp-Lieb data are unique (when they exist). The above diagram, for example, remains a valid presentation when the spaces $V_i$ are redefined in such a way that they increase in directions five and six before expanding in directions $e_1,\ldots,e_4$; this puts the diamond structure at the start of the graph rather than the end. It is natural to wonder if this is an indication that some simple or more canonical presentation is possible. Would it be possible, for example, to always build a valid presentation from a graph decomposition $\graph$ which is simply connected, as was the case for the Loomis-Whitney example? This example shows that the answer is ``no.'' This particular H\"{o}lder-Brascamp-Lieb inequality has no valid presentation of that form, and moreover the underlying H\"{o}lder-Brascamp-Lieb inequality cannot be deduced via interpolation from inequalities with such presentations.

Consider interpolation. Taking each $\tau_i$ to equal $1/2$ is in fact the only collection of exponents for these maps which give a finite value of the corresponding H\"{o}lder-Brascamp-Lieb constant, so no interpolation of any kind would be possible. To see this, suppose this ensemble of maps satisfies a H\"{o}lder-Brascamp-Lieb inequality for some unknown exponents $\tau_1,\tau_2,\tau_3,\tau_4$.   By the theorem of Bennett, Carbery, Christ, and Tao, we have $1 \leq \dim \spn \{e_1\} \leq \tau_1 + \tau_4$.  Likewise, by analyzing the projections of $e_2,e_3,e_4$, it must be true that $1 \leq \min \{\tau_1+\tau_4,\tau_1 + \tau_2, \tau_2 + \tau_4, \tau_3 + \tau_4\}$. But furthermore, \eqref{bcct2} implies
\begin{align*}
6 & = \dim \R^6 = 3 (\tau_1 + \tau_2) + 2 \tau_3 + 4 \tau_4 \\
 & = 2(\tau_1 + \tau_2) + (\tau_1 + \tau_4) + (\tau_2 + \tau_4) + 2 (\tau_3 + \tau_4).
 \end{align*}
From this, one sees that $\tau_1 + \tau_2 = \tau_1 + \tau_4 = \tau_2 + \tau_4 = \tau_3 + \tau_4 = 1$. By subtracting these identities from one another, it follows that $\tau_2 = \tau_4$, $\tau_1 = \tau_2$, $\tau_2 = \tau_3$, which forces $\tau_1 = \tau_2 = \tau_3 = \tau_4 = 1/2$.

Suppose now that for the given data, there is some valid presentation $(\graph,\{\theta_i\}_{i=1}^4)$ where $\graph$ has vertices $V_0',V_1',\ldots,V_6'$ with $V_0' := \{0\}$ and $V_6' := \R^6$ and its only edges are those edges joining $V_{i-1}'$ to $V_{i}'$ for $i=1,\ldots,6$. Because $\theta_1,\theta_2,\theta_3,\theta_4$ are balanced with total mass $1/2$, each must simply be constant and equal to $1/2$ on every edge. The summary weight $\sigma$ is also balanced and has total mass $1$, which must mean that on each edge, there are exactly two indices $i_1$, $i_2$ such that $\pi_{i_1}$ and $\pi_{i_2}$ project the endpoints of the edge to distinct subspaces. As a consequence, each edge also has two indices $i_1'$ and $i_2'$ such that both $\pi_{i_1'}$ and $\pi_{i_2'}$ project the endpoints of the edge to the same space.  Suppose that it is known for some $i$ that $V_i'$ is the span of some subcollection of $e_1,e_2,e_3,e_4$ (and note this is vacuously the case for $V_0'$). One can easily check that
\begin{align*}
(V_i' + \ker \pi_1) \cap (V_i' + \ker \pi_2) & = V_i' + \spn \{e_4\}, \\
(V_i' + \ker \pi_1) \cap (V_i' + \ker \pi_3) & = V_i' + \spn \{e_3\}, \\
(V_i' + \ker \pi_1) \cap (V_i' + \ker \pi_4) & = V_i', \\
(V_i' + \ker \pi_2) \cap (V_i' + \ker \pi_3) & = V_i' + \spn \{e_1\}, \\
(V_i' + \ker \pi_2) \cap (V_i' + \ker \pi_4) & = V_i', \\
(V_i' + \ker \pi_3) \cap (V_i' + \ker \pi_4) & = V_i' + \spn \{e_2\}.
\end{align*}
Assuming $i < 6$, let $v_{i+1}'$ be any vector in $V_{i+1}' \setminus V_{i}$. Since there must be exactly two indices $i_1'$ and $i_2'$ such that $\pi_{i'_1}$ and $\pi_{i_2'}$ project $V_i'$ and $V_{i+1}'$ to the same subspace, it must be the case that $v_{i+1}' + V_{i}'$ intersects $\ker \pi_{i_1'}$ and $\ker \pi_{i_2'}$ for some unique pair $(i_1',i_2')$ of distinct indices in $\{1,\ldots,4\}$. This forces $v_{i+1}'$ to belong to the intersection $(V_i' + \ker d \pi_{i_1'}) \cap (V_i' + \ker d \pi_{i_2'})$, which, by the computations above, means that modulo $V_i'$, the vector $v_{i+1}'$ must equal one of $e_1$,$e_2$,$e_3$, and $e_4$. This means that $V_{i+1}'$ is also the span of some subcollection of $e_1,e_2,e_3,e_4$. But by induction, there is, however, a contradiction for the space $V_5'$, since it must have dimension $5$ but is also contained in a space of dimension $4$, namely the span of $\{e_1,e_2,e_3,e_4\}$. Thus no valid presentation can have a graph $\graph$ with the asserted structure.

\section{Valid presentations imply finiteness}

Here we establish that the existence of a valid presentations implies \eqref{brascamplieb} for some $C < \infty$. This is accomplished by establishing basic properties of valid presentations and then associating a factorization to each valid presentation which is shown via elementary means to imply \eqref{brascamplieb}.

\subsection{Properties of graph decompositions and balanced weights}

The most important consequence of the definition of a graph decomposition is that every edge $e \in \edges(\graph)$ belongs to some maximal directed chain of edges which originates at $\{0\}$ and ends at $H$. (A \textit{chain of edges} is an ordered collection of edges $e_1,\ldots,e_k$ such that the vertex to which $e_i$ is incoming is the same vertex from which $e_{i+1}$ is outgoing.) This is because the dimension of vertices is strictly monotone along any directed chain and extension of a chain is possible whenever the endpoints of the chain are not $\{0\}$ and $H$. The proposition below demonstrates that balanced weights can be appropriately understood as linear combinations of weights which are constants on chains. 

\begin{proposition}
Suppose $\graph$ is a graph decomposition of $H \neq \{0\}$ and $\varphi$ is some balanced weight on the edges of $\graph$ with values in $[0,\infty)^N$ for some $N$. Then for some $k \geq 1$, there exist real-valued weight functions $\varphi_i$, $i=1,\ldots,k$, each of which is one on some maximal directed chain in $\graph$ and zero elsewhere, and values $c_i \in [0,\infty)^N$, $i=1,\ldots,k$, such that
 \label{propdecomp}
\begin{equation} \varphi = \sum_{i=1}^k c_i \varphi_i. \label{convex} \end{equation}
Additionally, the total mass $\tau$ of $\varphi$ equals $\sum_{i=1}^{k} c_i$, which further implies that the sum of weights of outgoing edges from $\{0\}$ equals the sum of weights incoming to $H$.  
\end{proposition}
\begin{proof}
First observe that any function $\varphi_i$ which is an indicator function of a maximal directed chain must itself be balanced with total mass one, since the weight of outgoing edges from $\{0\}$ is one, and at any vertex not equal to $\{0\}$ or $H$, either there are no incoming or outgoing vertices with nonzero weight, or there is exactly one incoming and one outgoing of nonzero weight, both with weight $1$.

If $\varphi$ is identically zero, there is nothing to prove (simply take $k=1$, fix $\varphi_1$ to be the indicator function of any maximal chain, and set $c_1 = 0$).  If $\varphi$ is not identically zero, there must be some edge $e$ on which $\varphi(e) \neq 0$, meaning that there is some $j$ such that the $j$-th component of $\varphi(e)$ is strictly positive. Let $e$ and $j$ be chosen minimally, i.e., the $j$-th component of $\varphi(e)$ has some positive value $\delta$ with the property that every positive component of $\varphi(e')$ for every edge $e'$ is at least $\delta$. Now let $C$ be any maximal chain of edges in $\graph$ passing through $e$ which has the property that the $j$-th component of $\varphi(e')$ is nonvanishing for all $e' \in C$. Because the $j$-th component of $\varphi$ must be balanced as a scalar-valued weight function on the edges
 and because it is never negative, any vertex other than $\{0\}$ and $H$ with an incoming (or outgoing) edge with positive $j$-th component must have a corresponding outgoing (or incoming) edge with positive $j$-th component. Thus the chain $C$ must begin at $\{0\}$ and end at $H$. Moreover, the difference $\varphi - c \varphi_C$, where $\varphi_C$ is the indicator function of $C$ and $c \in [0,\infty)^N$ equals $\delta$ in position $j$ and zero in all remaining positions, is 
 balanced, takes values in $[0,\infty)^N$ (because minimality of $\delta$ ensures that no component of $\varphi - c \varphi_C$ is ever negative), and vanishes on the $j$-th component of $e$ while $\varphi$ itself did not. If the difference $\varphi - c \varphi_C$ is not identically zero, one can iterate this process, with a guarantee each time that the sum over $e \in \edges(\graph)$ of the number of vanishing components of $\varphi(e)$ is strictly increasing at each step. This means that by some point $k$, one has constructed $\varphi_1,\ldots,\varphi_k$ and $c_1,\ldots,c_k \in [0,\infty)^N$ such that
\[ \varphi - \sum_{i=1}^k c_i \varphi_i \]
vanishes identically, as promised by \eqref{convex}. The total mass of $\sum_{i=1}^k c_i \varphi_i$ is simply $\sum_{i=1}^k c_i$ (because total mass is linear and the total mass of each $\varphi_i$ is $1$), so this must equal the total mass $\tau$ of $\varphi$. Likewise, the sum of each $\varphi_i$ over all edges outgoing from $\{0\}$ equals the sum over all edges incoming to $H$. 
\end{proof}

The next fundamental observation concerns projecting graph decompositions.
Given a linear map $\pi : H \rightarrow W$, any graph decomposition $\graph$ of $H$ induces a graph decomposition of $\pi(H)$, which will be denoted by $\pi(\graph)$ and is constructed in the following way:
\begin{enumerate}
\item The vertices $\verts(\pi(\graph))$ are exactly those subspaces $V$ of $\pi(H)$ having the form $V = \pi(V')$ for some $V' \in \verts(\graph)$.
\item Two unequal vertices $V_1,V_2 \in \verts(\pi(\graph))$ are joined by a directed edge if and only if there is an edge $e' \in \edges(\graph)$ joining $V_1'$ to $V_2'$ such that $\pi(V_1') = V_1$ and $\pi(V_2') = V_2$. In such a case, we say that $\pi(e') = e$.
\end{enumerate}
These rules fully specify the vertices and edges of $\pi(\graph)$. There are three main parts to the argument that $\pi(\graph)$ is a graph decomposition of $\pi(H)$. The first is to note that $\{0\}$ and $\pi(H)$ are vertices in $\pi(\graph)$ by virtue of being projections of $\{0\}$ and $H$ via $\pi$. Next, we establish that vertices joined by an edge satisfy the appropriate containment and dimensional relationships.
If $e$ is an edge outgoing from $V_1 \in \verts(\pi(\graph))$ and incoming to $V_2$, then by definition of $\edges(\pi(\graph))$, $V_1 \neq V_2$. Moreover, there is some edge $e' \in \edges(\graph)$ such that $e'$ is outgoing from $V_1'$ and incoming to $V_2'$ such that $\pi(V_1') = V_1$ and $\pi(V_2') = V_2$. Because $V_1'$ is a proper subspace of $V_2'$, $V_1$ must be a subspace of $V_2$ (and must be proper because $V_1 \neq V_2$). Because $V_2'$ is the sum of $V_1'$ and some one-dimensional subspace $X$, $V_2$ is the sum of $V_1$ and $\pi(X)$, meaning $\dim V_2 \leq \dim V_1 + 1$. These facts combined imply that $\dim V_2 = \dim V_1 + 1$.

Lastly, we show that every vertex other than $\{0\}$ has an incoming edge and every vertex other than $\pi(H)$ has an outgoing edge. To accomplish this, we first prove a useful auxiliary result. Suppose $e_1,\ldots,e_k$ is some directed chain of edges from $\{0\}$ to $H$ in $\graph$. Defining $e_{i_1},\ldots,e_{i_\ell}$ for $i_1 < \cdots < i_\ell$ to be all those edges among $e_1,\ldots,e_k$ whose endpoints project to distinct subspaces of $\pi(H)$, we claim that $\pi(e_{i_1}),\ldots,\pi(e_{i_\ell})$ is a directed chain in the graph $\pi(\graph)$ which begins at $\{0\}$ and ends at $\pi(H)$.
For any pair $e_{i_j}$ and $e_{i_{j+1}}$ in this list, if $i_{j+1} \neq i_j + 1$, then all edges $e_{i'}$ between $e_{i_j}$ and $e_{i_{j+1}}$ have endpoints which project to the same subspace via $\pi$. By transitivity, this means that the vertex to which $\pi(e_{i_j})$ is incoming must be the same vertex from which $\pi(e_{i_{j+1}})$ is outgoing. Thus $\pi(e_{i_1}),\ldots,\pi(e_{i_\ell})$ is a directed chain of edges in $\pi(\graph)$. Moreover, if $i_1 \neq 1$, endpoints of $\pi(e_{i'})$ both project to the same subspace of $\pi(H)$ via $\pi$ when $i' < i_1$, meaning by transitivity that they all project to $\{0\}$ and so $\pi(e_{i_1})$ must be outgoing from $\{0\}$. Similarly, $\pi(e_{\ell})$ is incoming to $\pi(H)$, since any edge $e_{i'}$ with $i' > i_\ell$ must have both endpoints projecting to $\pi(H)$. 

Returning to the issue of vertices and incoming and outgoing edges in $\pi(\graph)$, let $V$ be any vertex in $\pi(\graph)$. This $V$ has the form $\pi(V')$ for some $V' \in \verts(\graph)$. Assuming $H \neq \{0\}$ (as there would be nothing to prove in this case), $V'$ has either an incoming or outgoing edge, which may then be extended to a maximal directed chain in $\graph$. Projecting this chain back down to $\pi(\graph)$ in the manner just described gives a directed chain beginning at $\{0\}$, passing through $V$, and terminating at $\pi(H)$. In particular, every $V \neq \{0\}$ must have an incoming edge, and every $V \neq \pi(H)$ must have an outgoing edge.

When $\varphi$ is a scalar- or vector-valued function on $\edges(\graph)$, we define $\pi(\varphi)$ on $\edges(\pi(\graph))$ by taking
\[ \pi(\varphi)(e) := \sum_{\substack{e' \in \edges(\graph) \\ \pi(e') = e}} \varphi(e'), \]
i.e., by summing $\varphi$ over all edges which project via $\pi$ to the edge $e$. The following proposition establishes that balanced weights project to balanced weights.
\begin{proposition}
If $\graph$ is a graph decomposition of $H$ and $\varphi$ are balanced edge weights on $\graph$ with total mass $\tau$, then for any linear map $\pi : H \rightarrow W$, the projected weights $\pi(\varphi)$ on the graph decomposition $\pi(\graph)$ of $\pi(H)$ are also balanced with total mass $\tau$. \label{propproject}
\end{proposition}
\begin{proof}
Consider the case when $\varphi$ is one on some maximal directed chain in $\graph$ and zero elsewhere. As observed above, the edges $e'$ in this chain project via $\pi$ to a chain in $\pi(\graph)$. In this case, $\pi(\varphi)$ will be one on the elements of the projected chain and zero elsewhere since no two edges in the original chain can project nontrivially to the same edge in $\pi(\graph)$ (along the projected edges, the dimension of vertices is nondecreasing). Because $\pi(\varphi)$ is supported on and constant on a maximal chain from $\{0\}$ to $\pi(H)$, it is balanced and one can explicitly see that its total mass is $1$. 

Now by Proposition \ref{propdecomp}, any balanced weights can be written as a linear combination of weights on maximal chains. Each indicator function of a maximal chain projects to a balanced weight on $\pi(\graph)$ with the same mass, so by linearity of projection $\pi(\varphi)$, the proposition must follow.
\end{proof}

\subsection{Factorizations}

Given a graph decomposition $\graph$ of a finite-dimensional real Hilbert space $H$, an edge $e \in \edges(\graph)$ outgoing from $V_1$ and incoming to $V_2$, and a nonnegative Borel-measurable function $f$ on $H$, we will say that $f$ is a \textit{normalized edge function} for $e$ when two conditions hold. First, for every $v \in V_1$, it must be the case that $f(x + v) = f(x)$ for all $x \in H$. Second, for every unit vector $w \in V_2 \cap V_1^\perp$, it must be the case that
\[ \int_{\R} f(x + tw) dt \leq 1 \]
for every $x \in H$. (Note that $w$ is actually unique up to a choice of sign.) Normalized edge functions arise naturally in the theory of factorization for multilinear duality developed by Carbery H\"{a}nninen, and Valdimarsson \cite{chv2023}; one can find them, for example, in \cite{chv2023}*{Section 10.2.1}, along with a brief footnote explaining why functions of this sort are essentially unique for factorization purposes. As tools for proving inequalities, the chief obstacle is not the form these functions take, but the identification of appropriate subspaces to which these functions should be adapted. The structure encoded in a valid presentation generalizes Barthe's concept of adapted partitions \cite{barthe1998} and gives precisely information needed to build a factorization of general measurable functions using normalized edge functions. The first basic building block is the following proposition.

\begin{proposition} \label{factorprop} Suppose $\graph$ is a graph decomposition of $H$. For each nonnegative Lebesgue-measurable function $f$ on $H$, there exist nonnegative Borel-measurable functions $f_e$ for each $e \in \edges(\graph)$ such that each $f_e$ is a normalized edge function for $e$. These functions $f_e$ have the property that for any real, nonnegative balanced edge weights $\varphi$  on $\graph$ with total mass $\tau$, 
\begin{equation} [f(x)]^\tau = ||f||_1^{\tau} \prod_{e \in \edges(\graph)} [f_e(x)]^{\varphi(e)} \label{factorize} \end{equation}
for almost every $x \in H$. Here $||f||_1$ denotes the integral of $f$.
\end{proposition}
\begin{proof}
Modifying $f$ on a set of measure zero in $H$ does not change the validity of the factorization \eqref{factorize}, and in so doing, one may assume without loss of generality that $f$ is Borel rather than Lebesgue measurable. This has the advantage that $f$ remains Borel measurable when restricted to any affine subspace of $H$.
Given any $d$-dimensional subspace $V$ of $H$ and any nonnegative Borel-measurable function $f$ on $H$, we define
\[ f_V(x) := \int_{\R^d} f(x + t_1 w_1 + \cdots + t_d w_d) dt_1 \cdots d t_d \]
for any choice of orthonormal basis $w_1,\ldots,w_d$ of $V$. When $V= \{0\}$, we simply let $f_{\{0\}} := f$. Note that $f_V$ must itself be Borel measurable. This function $f_V$ is invariant under translation by elements of $V$, and for any $d'$-dimensional subspace $W$ of $H$ which has $W \cap V = \{0\}$, the function $(f_V)_W$ equals  $c_{V,W} f_{V+W}$ for some constant $c_{V,W}$ depending only on $V$ and $W$. In particular, if $W + V = H$, then $(f_V)_W$ is constant and equal to a factor depending only on $V$ and $W$ times the integral of $f$.

If $e$ is an edge in a graph decomposition $\graph$, let $V_1$ be the vertex from which $e$ is outgoing and $V_2$ the vertex to which $e$ is incoming and define
\[ f_{e}(x) := \frac{f_{V_1}(x)}{f_{V_2}(x)} \chi_{f_{V_2}(x) > 0}. \]
The function $f_e$ has the property that it is invariant with respect to translations in $V_1$. If $w$ is an unit vector in $V_2 \cap V_1^\perp$, we have that
\begin{align*}
\int f_{e}(x + tw) dt & = \int_{\R} \frac{f_{V_1}(x + tw )}{f_{V_2}(x + tw)} \chi_{f_{V_2}(x + tw) > 0} dt \\
& = \int_{\R} \frac{f_{V_1}(x + tw )}{f_{V_2}(x)} \chi_{f_{V_2}(x) > 0} dt \\
& = \frac{\chi_{f_{V_2}(x) > 0}}{{f_{V_2}(x)}} \int_{\R} {f_{V_1}(x + tw)}  dt \\ & = \chi_{f_{V_2}(x) > 0} \frac{f_{V_2}(x)}{f_{V_2}(x)} \leq 1
\end{align*}
for all $x \in H$.
If $\varphi$ is any nonnegative balanced function of edges in $\graph$ having total mass $\tau$, then $f$ has a factorization
\[ [f(x)]^\tau = ||f||_{1}^{\tau} \prod_{e \in \edges(\graph)} [f_e(x)]^{\varphi(e)} \]
which holds for almost every $x$, where $f_e$ is exactly as just defined. For definiteness, note that we interpret $[f_e(x)]^{\varphi(e)}$ to be identically one if $\varphi(e) = 0$. (Also note that if $||f||_1 = 0$, we can simply define $f_e$ to be identically zero for each $e$). The reason this factorization holds is simply that the definition of $f_e$ yields the identity
\begin{align*}
\prod_{e \in \edges(\graph)} & \left( f_e(x) \right)^{\varphi(e)} = \\ & \!\!\!\!\! \prod_{V \in \verts(\graph)} \left( \prod_{e \in \edgeout(V)} [f_V(x)]^{\varphi(e)} \right) \left( \prod_{e \in \edgein(V)} [f_V(x)]^{-\varphi(e)} (\chi_{f_V(x) > 0})^{\varphi(e)} \right),
\end{align*}
where $\edgeout(V)$ denotes edges outgoing from $V$ (i.e., edges for which $V$ has smaller dimension than the space joined to it by $e$) and $\edgein(V)$ are the edges incoming to $V$. Thus if we let
\[ I(x) := \prod_{V \in \verts(\graph)} \prod_{e \in \edgein(V)} (\chi_{f_V(x)>0})^{\varphi(e)}, \]
we have that
\begin{align*} \prod_{e \in \edges(\graph)} [f_e(x)]^{\varphi(e)} & = I(x) \prod_{V \in \verts(\graph)} [f_V(x)]^{\sum_{e \in \edgeout(V)} \varphi(e) - \sum_{e \in \edgein(V)} \varphi(e)} \\ & = I(x) [f(x)]^{\tau} ||f||_{1}^{-\tau}. \end{align*}
Aside from the function $I(x)$, this is exactly the desired factorization. To see why $I(x)$ can be ignored almost everywhere, first note that it equals $1$ or $0$ everywhere, so one need only consider points $x$ at which it vanishes. For fixed $V$, let $E$ be the set of points at which $f(x)$ is nonvanishing and $f_V(x)$ vanishes.  The function $(f \chi_E)_V$ is identically zero because it is dominated by $f_V(x)$, which is constant on translates of $V$, and $E$ contains no translates of $V$ where the integral of $f$ is nonzero. Thus Fubini's Theorem guarantees that $f \chi_E$ vanishes almost everywhere. Therefore, aside from a null set, $f$ must vanish everywhere that $f_V$ vanishes. Taking products over $V \in \verts(\graph)$ establishes that $f(x)$ vanishes almost everywhere on the set where $I(x) = 0$, so the proof of the factorization is complete.
\end{proof}

The second main building block for a factorization argument is the estimation step, in which one bounds the norm of a function admitting a factorization. This is expressed in Proposition \ref{integralprop}.
\begin{proposition}
Suppose $F$ is a nonnegative measurable function on $H$ and that $\graph$ is a graph decomposition of $H$ with balanced edge weights $\varphi$ having total mass $1$. If $F$ admits a factorization
\[ F(x) = \prod_{e \in \edges(\graph)} [F_e(x)]^{\varphi(e)} \]
for a.e. $x \in H$, where each $F_e$ is a normalized edge function for $e$, then \label{integralprop}
\[ \int_{H} F(x) dx \leq 1. \]
\end{proposition}
\begin{proof}
By Proposition \ref{propdecomp},  $\varphi$ can be written as a linear combination
\[ \sum_{\ell=1}^k c_\ell \varphi_\ell \]
where $c_{\ell}$ are nonnegative real numbers summing to $1$ and each $\varphi_\ell$ is identically one on some maximal chain from $\{0\}$ to $H$ and zero elsewhere. If $\mathcal{C}_\ell$ is the collection of edges where $\varphi_\ell$ is supported, we have that 
\[ \prod_{e \in \edges(\graph)} [F_e(x)]^{\varphi(e)} = \prod_{\ell=1}^k \left[ \prod_{e \in \mathcal{C}_\ell} F_e(x) \right]^{c_\ell}. \]
Because the $c_\ell$ sum to $1$, we may apply H\"{o}lder's inequality:
\begin{equation} \int F(x) dx \leq \prod_{\ell=1}^k \left[ \int \prod_{e \in \mathcal{C}_\ell} F_e(x) dx \right]^{c_\ell}. \label{boundprod} \end{equation}
To prove the proposition, then, it suffices to assume that $\varphi$ is the indicator function of a maximal directed chain (and thereby show that each term in the product over $\ell$ on the right-hand side of \eqref{boundprod} is at most $1$).

To reiterate, it suffices to show that when $e_1,\ldots,e_m$ is a directed chain which begins at $\{0\}$ and ends at $H$ and $F_\ell$ is a normalized edge function for $e_\ell$ for each $\ell=1,\ldots,m$, then 
\[ \int_{H} \prod_{\ell=1}^m F_\ell(x) dx \leq 1. \]
 The proof is a simple consequence of Fubini's Theorem. Let $V_{\ell-1}$ be the vertex from which $e_\ell$ is outgoing and let $V_{\ell}$ be the vertex to which $e_\ell$ is incoming.  There is an orthonormal basis of $H$ such that $\{w_1,\ldots,w_\ell\}$ is a basis of $V_\ell$ for each $\ell$. This basis is unique, in fact, up to choice of $\pm$ sign for each $w_i$. Then write
\begin{align*}
\int \prod_{\ell=1}^m & F_{\ell}(x) dx  = \int \prod_{\ell=1}^m F_{\ell} (t_1 w_1 + \cdots + t_m w_m) dt_1 \cdots dt_m
\end{align*}
and observe that $F_\ell$ being constant on $V_{\ell-1}$ implies that
$F_\ell(t_1 w_1 + \cdots + t_m w_m ) = F_\ell(t_\ell w_\ell + \cdots + t_m w_m)$, i.e., it is independent of $t_1,\ldots,t_{\ell-1}$. Let $*_\ell$ denote the sum $t_{\ell+1} w_{\ell+1} + \cdots + t_m w_m$ when $\ell < m$ and let $*_m = 0$. The key feature of $*_\ell$ is that it does not depend on $t_{\ell'}$ when $\ell' \leq \ell$.  There is a triangular structure to the integral which can be iterated:
\begin{align*}
& \int \prod_{\ell=1}^m  F_{\ell}(x) dx  =  \int \prod_{\ell=1}^m F_{\ell} (t_\ell w_\ell + *_\ell) dt_1 \cdots d t_m \\
& = \int \left[ \int F_{1} (t_1 w_1 + *_1) dt_1 \right] \prod_{\ell=2}^m F_{\ell} (t_\ell w_\ell + *_\ell) dt_2 \cdots d t_m \\
& \leq \int \prod_{\ell=2}^m F_{\ell} (t_\ell w_\ell + *_\ell) dt_2 \cdots d t_m  \leq \cdots  \\
& \leq \int \left[ \int F_j (t_j w_j + *_j) dt_2 \right] \prod_{\ell=j+1}^{N} F_{\ell}(t_\ell w_\ell +*_\ell) d t_{j+1} \cdots d t_m \leq \cdots \leq 1. 
\end{align*}
This completes the proof of the proposition.
\end{proof}

\subsection{Finiteness of the constant}

We are now ready to prove the main direction of Theorem \ref{mainthm}, namely, that existence of valid presentations implies \eqref{brascamplieb} for some $C < \infty$.

Given linear maps $\pi_i : H \rightarrow H_i$ for $i=1,\ldots,N$, it may be assumed without loss of generality that $\dim \pi_i(H) > 0$ for each $i$, since otherwise $f_i^{\tau_i}(\pi_i(x))$ is simply a constant function which equals $(\int_{\pi_i(H)} f)^{\tau_i}$ when the measure on $\{0\}$ is simply the delta measure with mass $1$.  Suppose that there exists a valid presentation $(\graph,\{\theta_i\}_{i=1}^N)$ of the H\"{o}lder-Brascamp-Lieb data $(H,\{\pi_i\}_{i=1}^N,\{\tau_i\}_{i=1}^N)$.
For each $i=1,\ldots,N$, there is a graph decomposition $\pi_i(\graph)$ of $\pi_i (H)$ (and, by assumption, $\pi_i(H)$ is not trivial) and the projected weights $\pi_i(\theta_i)$ are, by Proposition \ref{propproject} balanced with total mass $\tau_i$ on $\pi_i(\graph)$.
For each $i=1,\ldots,N$ and each nonnegative Lebesgue-measurable function $f_i$ on $\pi_i(H)$, apply the factorization from Proposition \ref{factorprop} using $\pi_i(\graph)$ and $\pi_i(\theta_i)$:
\begin{align*}
 [f_i(u)]^{\tau_i} & = ||f_i||_1^{\tau_i} \prod_{e' \in \edges(\pi_i(\graph))} [f_{i,e'}(u)]^{\pi_i(\theta_i)(e')} \\ & = ||f_i||_1^{\tau_i} \prod_{e' \in \edges(\pi_i(\graph))} \prod_{\substack{e \in \edges(\graph) \\ \pi_i(e) = e'}} [f_{i,e'}(u)]^{\theta_i(e)} \\
& =  ||f_i||_1^{\tau_i}  \prod_{\substack{e \in \edges(\graph) \\ \pi_i(e) \text{ defined}}} [f_{i,\pi_i(e)}(u)]^{\theta_i(e)}
 \end{align*}
 almost everywhere, 
where $\pi_i(e)$ is understood to be defined exactly when vertices of $e$ project to distinct subspaces via $\pi_i$ and is undefined otherwise.
Then set $u = \pi_i(x)$ and take the product over $i$:
\begin{align*}
\prod_{i=1}^N [f_i(\pi_i(x))]^{\tau_i} & = \prod_{i=1}^N \left[ ||f_i||_{1}^{\tau_i} \prod_{\substack{e \in \edges(\graph) \\ \pi_i(e) \text{ defined}}} [f_{i,\pi_i(e)}(\pi_i(x))]^{\theta_i(e)} \right] \\
& = \left( \prod_{i'=1}^N ||f_{i'}||_{1}^{\tau_{i'}} \right) \prod_{e \in \edges(\graph)} \left[ \prod_{\substack{i=1,\ldots,N \\ \pi_i(e) \text{ defined}}} [f_{i,\pi_i(e)}(\pi_i(x))]^{\theta_i(e)} \right]
\end{align*}
for a.e. $x \in H$. When $e$ is an edge from $V_1$ to $V_2$ and $\pi_i(e)$ is defined, the function $f_{i,\pi_i(e)}(x)$ is constant along translates of $\pi_i(V_1)$. This means that $f_{i,\pi_i(e)}(\pi_i(x))$ is constant along translates of $V_1$. Likewise, if $w \in V_2 \cap V_1^\perp$ is a unit vector, then $\pi_i(w) = w_0 + c_i w_1$ for some $w_0 \in \pi_i(V_1)$ and some unit vector $w_1 \in \pi_i(V_2) \cap (\pi_i(V_1))^\perp$. It follows that
\[ \int f_{i,\pi_i(e)}(\pi_i( x + t w)) dt = \int f_{i,\pi_i(e)} (\pi_i(x) + t c_i w_1) dt \leq c_i^{-1} \]
for almost every $x \in H$. For emphasis: the constant $c_i$ equals the constant
\[ ||\pi_i \cdot e || := \sup_{\substack{w \in V_2 \\ ||w|| = 1}} || P_{\pi_i(V_1)}^{\perp} w ||, \]
where $P_{\pi_i(V_1)}^\perp$ is projection onto the orthogonal complement of $\pi_i(V_1)$. 

Because $||\pi_i \cdot e|| f_{i,\pi_i(e)} \circ \pi_i$ is nonnegative and Borel measurable, we have established that, for each $i$ such that $\pi_i(e)$ is defined, the function $||\pi_i \cdot e || f_{i, \pi_i(e)} \circ \pi_i$ is a normalized edge function for $e$.
It follows by H\"{o}lder's inequality that if we define summary weights $\sigma (e)$ to be the sum of $\theta_i(e)$ for all $i$ such that $\pi_i(e)$ is defined, then
\[ F_e(x) := \prod_{\substack{i=1,\ldots,N \\ \pi_i(e) \text{ defined}}}  [||\pi_i \cdot e || f_{i,\pi_i(e)} (\pi_i(x))]^{\theta_i(e)/\sigma(e)} \]
 is a normalized edge function for $e$ and we have a factorization
\begin{equation} \prod_{i=1}^n [f_i(\pi_i(x))]^{\tau_i} = C \left( \prod_{i=1}^N ||f_i||_1^{\tau_i} \right) \prod_{e \in \edges(\graph)} [F_e(x)]^{\sigma(e)} \label{almostdone} \end{equation}
for almost every $x$, where 
\begin{equation} C := \prod_{i=1}^N \prod_{\substack{e  \in \edges(\graph) \\ \pi_i(e) \text{ defined}}} ||\pi_i \cdot e||^{- \theta_i(e)}. \label{boundvalue} \end{equation}
 Applying Proposition \ref{integralprop} using the weights $\sigma$ on $\edges(\graph)$, which are balanced with total mass $1$, it follows that
\[ \int_{H} \prod_{i=1}^n [f_i(\pi_i(x))]^{\tau_i} dx \leq  C \prod_{i=1}^N \left( \int_{\pi_i(H)} f_i dx \right)^{\tau_i}  \]
for the same constant $C$ defined in \eqref{boundvalue}. This implies in particular that the H\"{o}lder-Brascamp-Lieb constant is finite and not larger than $C$.

\section{Existence of valid presentations}
\label{exist}

In this section, we prove that \eqref{brascamplieb} can hold with  finite constant $C$  only when there exists a valid presentation $(\graph,\{\theta_i\}_{i=1}^N)$. By the results of the previous section, the existence of such structures implies that \eqref{brascamplieb} holds. By the result of Bennet, Carbery, Christ, and Tao \cites{bcct2008,bcct2010}, we know that finiteness of the constant $C$ in \eqref{brascamplieb} implies that the exponents $\{\tau_i\}_{i=1}^N$ satisfy the inequalities \eqref{bcct1}
with equality when $V = H$. It therefore suffices to show that any tuple $(\tau_1,\ldots,\tau_N)$ satisfying the constraints \eqref{bcct1} and \eqref{bcct2} must admit a valid presentation $(\graph,\{\theta_i\}_{i=1}^N)$ for the data $(H,\{\pi_i\}_{i=1}^N,\{\tau_i\}_{i=1}^N)$.

The proof is by induction on dimension of $H$ and uses the same ideas employed both by both Bennet, Carbery, Christ, and Tao and Carbery, H\"{a}nninen, and Valdimarsson to show that \eqref{bcct1} and \eqref{bcct2} imply \eqref{brascamplieb}. When $\dim H = 1$, the inequality \eqref{brascamplieb} can hold only when 
\begin{equation} 1 = \sum_{i=1}^N \tau_i \dim \pi_i(H), \label{surprisetm} \end{equation}
In this case, it is straightforward to manufacture a valid presentation by hand. Let $\graph$ consist of a single edge from $\{0\}$ to $H$. Let $(\theta_1(e),\ldots,\theta_N(e))$ equal $(\tau_1,\ldots,\tau_N)$ on the single edge $e \in \edges(\graph)$.  Because there is only one edge, all weights are automatically balanced. Note that the summary weight $\sigma$ is exactly the sum of $\tau_i$ over those $i$ for which $\dim \pi_i(H) = 1$, so \eqref{surprisetm} guarantees that the total mass of the summary weight $\sigma$ is $1$. Therefore, when $\dim H =1$ and \eqref{brascamplieb} is true, a valid presentation exists.

In general, we note that valid presentations have a natural convexity property: If $\graph$ admits weights $\theta$ satisfying the hypotheses of Theorem \ref{mainthm} and $\graph'$ and $\theta'$ do as well for the same underlying projections, we may construct a new graph $\graph''$ whose vertex set and edge set are each simply the union of vertex sets and edge sets for $\graph$ and $\graph'$ respectively, and then build $\theta''$ such that $\theta''(e) = c \theta(e) + (1-c) \theta'(e)$, where we interpret $\theta(e)$ and $\theta'(e)$ as being zero when either is not defined for a particular edge $e$. Assuming $c \in [0,1]$, this $(\graph'',\{\theta''_i\}_{i=1}^N)$ is a valid presentation of $(H,\{\pi_i\}_{i=1}^N,\{c \tau_i + (1-c)\tau_i'\}_{i=1}^N)$: $c \theta_i + (1-c) \theta_i'$ is balanced and nonnegative with total mass $c \tau_i + (1-c ) \tau_i'$ for each $i$ and the summary weight $\sigma''$ equals $c \sigma + (1-c) \sigma'$, which is balanced and nonnegative with total mass $1$.

By this convexity of valid presentations, to show the existence of a valid presentation for all tuples $(\tau_1,\ldots,\tau_N)$ for which \eqref{brascamplieb} holds with finite constant, it suffices to assume that $(\tau_1,\ldots,\tau_N)$ is an extreme point of the compact, convex set in $[0,1]^N$ containing all tuples satisfying the constraints \eqref{bcct1} and \eqref{bcct2}. Also recall the implicit constraints $\tau_i \geq 0$ and $\tau_i \leq 1$ for each $i=1,\ldots,N$. Considering the collection \eqref{bcct1}, \eqref{bcct2} augmented by these additional constraints, at any extreme point $(\tau_1,\ldots,\tau_N)$, there must be at least $N$ linearly independent constraints among the whole collection which are satisfied as equalities. This is because when the number of linearly-independent constraints which are satisfied as equalities is less than $N$, there will necessarily be some nonzero direction $(u_1,\ldots,u_N) \in \R^N$ such that all equality constraints satisfied by $(\tau_1,\ldots,\tau_N)$ remain satisfied by  $(\tau_1 + s u_1,\ldots, \tau_N + s u_N)$ for all $s \in \R$. Since the remaining constraints are all strict inequalities, they remain valid for sufficiently small $s$, and this shows that $(\tau_1,\ldots,\tau_N)$ is not extremal after all (because we could interpolate these inequalities for small positive and negative values of $s$ to deduce the case $s=0$). So there must be some collection $S \subset \{1,\ldots,N\}$ such that $\tau_i \in \{0,1\}$ for all $i \in S$ and there must  be at least $N - \# S$ subspaces $V$ on which \eqref{bcct1} holds as an equality. Note that we know that $H$ itself is one such space. The subspaces $V \neq \{0\},H$ for which \eqref{bcct1} holds as equality are known as critical subspaces, a term first introduced by Carlen, Lieb, and Loss \cite{cll2004} in the special case where each $\pi_i$ is rank $1$.

If $\tau_i = 1$ for any $i$ for which $\dim \pi_i(H)  > 0$, then automatically there is a critical subspace: let $W$ be any codimension $1$ subspace in $\pi_i(H)$ and let $V := \pi_i^{-1}(W)$. We know that
\[ \dim H = \sum_{i=1}^N \tau_i \dim \pi_i(H) \]
and also that
\begin{align*}
\dim H -1 = \dim V & \leq \sum_{i=1}^N \tau_i \dim \pi_i(V) \\ & \leq -1 + \sum_{i=1}^N \tau_i \dim \pi_i(H) = -1 + \dim H \end{align*}
because $\dim \pi_i(V) = \dim W = \dim \pi_i(H) - 1$. Thus, $V$ is a critical subspace.

Next, if $\tau_i = 0$ for all $i$ which $\dim \pi_i(H) > 0$ or all but one such value of $i$, then the conditions \eqref{bcct1}, \eqref{bcct2} cannot be satisfied. When all $\tau_i$ are zero, \eqref{bcct2} states that $0 < \dim H \leq 0$, and if all but one of the terms $\tau_i \dim \pi_i(H)$ vanishes, then applying \eqref{bcct1} to the kernel $K$ of the only $\pi_i$ for which $\tau_i \dim \pi_i(H)  > 0$ gives
\[ \dim K \leq \sum_{i' \neq i} \tau_i \cdot \dim \pi_{i'}(K) + \tau_i \cdot 0 = 0. \]
This is a contradiction unless $\dim K = 0$. But when $\dim K = 0$, the tuple $(\tau_1',\ldots,\tau_N') := (0,\ldots,1,0,\ldots,0)$, where $\tau_i = 1$ for the index $i$ such that $\ker \pi_i = \{0\}$, satisfies \eqref{bcct1} as an equality for each $V$. But then we can write the exponents as a convex combination
\begin{align*} (\tau_1,\ldots,\tau_N) = \tau_i & (0,\ldots,0,1,0,\ldots,0) \\ & + (1-\tau_i) \left(\frac{\tau_1}{1-\tau_i},\ldots,\frac{\tau_{i-1}}{1-\tau_i},0,\frac{\tau_{i+1}}{1-\tau_i},\ldots,\frac{\tau_N}{1-\tau_i}\right) \end{align*}
of two different $N$-tuples satisfying \eqref{bcct1} and \eqref{bcct2}, contradicting extremality.
Therefore, at every extreme point which is not already known to admit a critical subspace $V$, there must be at least two indices $i,i'$ such that $\dim \pi_i(H) > 0$, $\dim \pi_{i'}(H) > 0$ and $0 < \tau_i,\tau_i' < 1$. This means at least two different inequalities from \eqref{bcct1} must be equalities, and even though one of them is presumably \eqref{bcct2}, every extreme point must nevertheless always admits a proper critical subspace $V$ for which the constraint \eqref{bcct1} it generates holds as an equality.

Critical subspaces are significant because H\"{o}lder-Brascamp-Lieb inequalities hold when integrating over $V$ instead of $H$ because
\[ \dim V' \leq \sum_{i=1}^n \tau_i \dim \pi_i(V') \]
for all subspaces $V'$ of $V$, with equality when $V' = V$.  Similarly, for subspace $V' \supset V$, we have
\begin{equation} \dim V' - \dim V \leq \sum_{i=1}^N \tau_i ( \dim \pi_i(V') - \dim \pi_i(V)) \label{subtracted} \end{equation}
with equality when $V' = H$. This second set of inequalities is also associated to a H\"{o}lder-Brascamp-Lieb problem, namely: we replace $H$ by $V^{\perp}$ (the orthogonal complement of $V$)  and $\pi_i$ by the composition $P_i \pi_i$, where $P_i$ is orthogonal projection onto $[\pi_i(V)]^{\perp}$. This follows because every subspace $W$ of $V^\perp$ satisfies
\[\dim W = \dim (W + V) - \dim V \]
and 
\[ \dim P_i \pi_i(W) = \dim P_i \pi_i(W+V) = \dim \pi_i(W+V) - \dim \pi_i(V). \]
The inequalities \eqref{subtracted} reduce to \eqref{bcct1} and \eqref{bcct2} for this system by simply fixing $V' = W + V$ for each subspace $W \subset V^\perp$.
So by induction on dimension of $H$, for any tuple $(\tau_1,\ldots,\tau_N)$ admitting a critical subspace $V$, there are graph decompositions $\graph,\graph^\perp$ of the vector spaces $V,V^\perp$ and weights $\theta_i,\theta^\perp_i$ which are valid presentations for their corresponding H\"{o}lder-Brascamp-Lieb data. From these two structures, we build a single graph $\graph^+$:
\begin{enumerate}
\item The vertices of $\graph^+$ are those subspaces of $H$ which are either elements of $\verts(\graph)$ or have the form $V + W$ for some $W \in \verts(\graph^{\perp})$.
\item Vertices $V_1,V_2 \in \verts(\graph^+)$ are joined by an edge $e$ if and only if dictated by one of the two graphs $\graph$, $\graph^\perp$, i.e., when $V_1,V_2 \in \graph$ and are joined by an edge there or when $V_1 = V + W_1$ and $V_2 = V + W_2$ for $W_1,W_2  \in \verts(\graph^{\perp})$ which are joined by an edge there.
\end{enumerate}
We note that every edge belongs to a chain which extends from $\{0\}$ to $H$ because as graphs, both $\graph$ and $\graph^\perp$ are subgraphs of $\graph$ with the terminal vertex $V$ of $\graph$ being identified with the starting vertex of $\graph^\perp$. The graph $\graph^+$ essentially just concatenates paths in $\graph$ and $\graph^\perp$. Thus $\graph^+$ is indeed a graph decomposition of $H$.
Now define edge weights $\theta^+$ on $\graph^+$ as follows:
\begin{enumerate}
\item On an edge $e$ joining $V_1,V_2 \in \verts(\graph^+)$ which are both also joined by an edge $e' \in \edges(\graph)$, define $\theta^+_i(e)$ on $\graph^+$ to equal $\theta_i(e')$ as it was defined on $\graph$ for the corresponding edge $e' \in \edges(\graph)$.
\item On an edge $e$ joining $V_1,V_2$ having the form $V_1 = V + W_1$ and $V_2 = V + W_2$ for $W_1$ and $W_2$ joined via $e^\perp$ in $\graph^{\perp}$, let $\theta^+_i(e) := \theta_i^{\perp}(e^\perp)$ for the corresponding edge $e^\perp \in \edges(\graph^\perp)$.
\end{enumerate}
To verify that $\theta_i^+$ is balanced on $\graph^+$ with total mass $\tau_i$, it suffices to check the balance condition at $V$, because $V$ is the only vertex in $\graph^+$ whose incoming and outgoing edges aren't isomorphic to and weighted identically to a corresponding vertex of either $\graph$ of $\graph^\perp$. At $V$, the sum of $\theta_i^+$ on all incoming edges is $\tau_i$ because $\theta_i$ is balanced on $\graph$ with total weight $\tau_i$ and all incoming edges to $V$ are derived from incoming edges to $V$ in $\graph$. Similarly, the sum of $\theta_i^+$ on all outgoing edges from $V$ is also $\tau_i$ because $\theta_i^\perp$ is balanced on $\graph^{\perp}$ with total mass $\tau_i$ and the outgoing edges from $V$ are all derived from outgoing edges of $\{0\}$ in $\graph^\perp$.

Finally, consider the summary weight $\sigma^+$. Let $V_1,V_2$ be vertices in $\graph^+$ joined by an edge $e$ which is derived from a corresponding edge in $\graph$. The indices $i$ included in the sum for the summary weight remain unchanged in passing from $\graph$ to $\graph^+$ because restricting the domain of each $\pi_i$ to $V$ has no impact on the question of whether $\pi_i(V_1) = \pi_i(V_2)$ since $V_1,V_2 \subset V$. Thus $\sigma^+$ agrees with $\sigma$ on the subgraph isomorphic to $\graph$. Now suppose $V_1$ and $V_2$ are vertices joined by an edge in $\graph^+$ which is derived from an edge in $\graph^{\perp}$. Here $V_1 = V + W_1$ and $V_2 = V + W_2$ for $W_1,W_2 \subset V^\perp$. It is similarly the case that $\pi_i(V_1) = \pi_i(V_2)$ if and only if $P_i \pi(W_1) = P_i \pi(W_2)$ since equality of $P_i \pi(W_1)$ and $P_i \pi(W_2)$ is identical to the question of equality of $\pi(W_1) + \pi(V)$ and $\pi(W_2) + \pi(V)$, which are exactly the spaces $\pi_i(V_1)$ and $\pi_i(V_2)$, respectively. Thus, $\sigma^+$ agrees with $\sigma^\perp$ on the portion of $\graph^+$ which is isomorphic to $\graph^\perp$. Just as was argued a moment ago, because both $\sigma$ and $\sigma^\perp$ are balanced on their respective graphs with total mass $1$, this means that $\sigma^+$ is also balanced with total mass $1$. Thus $(\graph^+,\{\theta_i^+\}_{i=1}^N)$ is a valid presentation of the data $(H,\{\pi_i\}_{i=1}^N, \{\tau_i\}_{i=1}^N)$.

\section{Closing Comments}
We conclude with a few observations about the proof.
\begin{enumerate}
\item The proof in Section \ref{exist} also yields an upper bound on the number of vertices needed for a graph decomposition.  Let $V_{N,m}$ be the maximum number of vertices which appear in a so-constructed valid presentation of H\"{o}lder-Brascamp-Lieb data with $N$ mappings $\pi_i$ on a Hilbert space $H$ of dimension $m$. We explicitly have that $V_{N,1} = 2$. If there is a critical subspace, the concatenation procedure yields a graph decomposition with at most $V_{N,m'} + V_{N,m-m'} - 1$ vertices for some $1 \leq m' < m$. For a general tuple $(\tau_1,\ldots,\tau_N)$, one must take a convex combination of valid presentations generated at extremal points. Without loss of generality, one may assume that at most $N+1$ extremal points are included in this convex combination. Thus one needs at most $(N+1) [V_{N,m'} + V_{N,m-m'} - 3] + 2$ vertices for a general point (where the extra $-2$ and $+2$ come from the fact that $\{0\}$ and $H$ are common to all graph decompositions). 
One can establish via these observations that there are at most $N^{-1}[(N+1)^{\dim H} - 1] + 1$ vertices in the graph decomposition $\graph$ generated by this construction. This bound seems unlikely to be sharp.
\item One can verify that the constant $C$ given by \eqref{boundvalue} is independent of the Hilbert space structure, meaning that changing the inner products on the spaces $H$ and $\pi_i(H)$ preserves the value of the constant $C$ provided that the volume of the unit ball is preserved (so that the induced integrals over $H$ and $\pi_i(H)$ have unchanged normalization). For each edge $e \in \edges(\graph)$, let $e_{-}$ be the vertex from which $e$ is outgoing and $e_{+}$ be the vertex into which it is incoming and choose some $w_e \in e_{+} \setminus e_{-}$. It follows from the definition of $||\pi_i \cdot e||$ that 
\[ ||\pi_i \cdot e|| = \frac{||P^\perp_{\pi_i(V_1)} \pi_i(w_e)||}{||P^\perp_{V_1} w_e||}. \]
Substituting this into \eqref{boundvalue} allows one to reduce to products over maximal chains using Proposition \ref{propdecomp} and the balance conditions (with a little extra care needed in cases when multiple edges $e \in \edges(\graph)$ happen to project to the same edge in one or more of the $\pi_i(\graph)$). On these chains, the products can be written in terms of determinants rather than norms, which explicitly manifests the claimed invariance property. This also suggests interesting potential connections between the expression for $C$ and invariant polynomials of the H\"{o}lder-Brascamp-Lieb data of the sort described in \cite{gressman2021} which are known abstractly to exist but are not currently known explicitly.
\item The explicit nature of the factorization \eqref{factorize} immediately hints at a number of generalizations. It appears very likely, for example, that a similar method of proof should be able to prove appropriately-structured mixed-norm H\"{o}lder-Brascamp-Lieb inequalities. In a different direction, one can see natural nonlinear variations where vertices of $\graph$ are foliations of the underlying space rather than vector spaces and edges $e \in \edges(\graph)$ join compatible foliations along which one could partially integrate. In cases where there is natural stability of the transversality of kernels (as, for example, in \cites{bb2010,bbg2013}), it should be possible to build useful foliations. In cases without such stability, as for example, the work of Bennett, Bez, Buschenhenke, Cowling, and Flock \cite{bbbcf2020}, it is far less clear whether such an approach would be feasible.
\end{enumerate}

\bibliography{mybib}

\end{document}